\documentclass[oneside,draft,11pt]{amsart}

\usepackage{dsfont}

\newcommand{\KS}{K_{\mathcal S}}
\newcommand{\tauS}{\tau_{\mathcal S}}
\newcommand{\omegaS}{\omega_1^{\mathcal S}}
\newcommand{\dd}{\mathrm d}
\newcommand{\C}{\mathcal C}

\DeclareMathOperator{\supp}{supp}

\title[Continuous self-maps of the ladder system space]{A note on the continuous self-maps of the ladder system space}

\author{Claudia Correa}
\thanks{The first author is sponsored by FAPESP (Processo n.\ 2010/17853-8).}
\address{Departamento de Matem\'atica,\hfill\break\indent Universidade de S\~ao Paulo, Brazil}
\email{claudiac.mat@gmail.com}
\author{Daniel V. Tausk}
\address{Departamento de Matem\'atica,\hfill\break\indent Universidade de S\~ao Paulo, Brazil}
\email{tausk@ime.usp.br} \urladdr{http://www.ime.usp.br/\~{}tausk}

\date{January 23rd, 2013}

\begin{document}

\theoremstyle{plain}\newtheorem{teo}{Theorem}
\theoremstyle{plain}\newtheorem{prop}[teo]{Proposition}
\theoremstyle{plain}\newtheorem{lem}[teo]{Lemma}
\theoremstyle{plain}\newtheorem{cor}[teo]{Corollary}
\theoremstyle{definition}\newtheorem{defin}[teo]{Definition}
\theoremstyle{remark}\newtheorem{rem}[teo]{Remark}
\theoremstyle{plain}\newtheorem{assum}[teo]{Assumption}
\theoremstyle{definition}\newtheorem{example}[teo]{Example}

\begin{abstract}
We give a partial characterization of the continuous self-maps of the ladder system space $\KS$. Our results show that
$\KS$ is highly nonrigid. We also discuss reasonable notions of ``few operators'' for spaces $C(K)$ with scattered $K$
and we show that $C(\KS)$ does not have few operators for such notions.
\end{abstract}

\maketitle

\begin{section}{Introduction}

The ladder system space $\KS$ is a standard example of a compact scattered space of finite (Cantor--Bendixson) height
(\cite[pg.\ 164]{Ark}). It was used by R. Pol \cite{Pol}
to obtain the first example of a weakly Lindel\"of Banach space $C(K)$ that is not WCG. The space $\KS$ is the one-point
compactification $\omega_1\cup\{\infty\}$ of the first uncountable ordinal $\omega_1$ endowed with the ladder system topology $\tauS$
(see Section~\ref{sec:main}). A fairly uninteresting type of continuous self-map $\phi$ of $\KS$ consists of those that
are ``almost constant'' in the sense that they have countable range (and thus yield ``small'' composition operators
$\C_\phi:f\mapsto f\circ\phi$ on $C(\KS)$ of separable range). Those include the continuous
maps that do not fix the point at infinity $\infty$. One might wonder whether $\KS$ is ``almost rigid'' in the sense that the continuous self-maps
of $\KS$ are either ``almost constant'' or are close to the identity, i.e., have lots of fixed points. We answer this question in the negative,
providing a simple example of a continuous self-map of $\KS$ having uncountable range and whose only fixed point is $\infty$
(Example~\ref{exa:movebastante}). We will show that
the existence of such a map already implies that $C(\KS)$ does not have few operators for some natural notions of
``few operators'' (Section~\ref{sec:operators}). Even though maps of uncountable range may fix only the point at infinity, we prove that
if $\phi:\KS\to\KS$ is a continuous map fixing $\infty$ then there must exist a club $F\subset\omega_1$ such that for all $\alpha\in F$ either
$\phi(\alpha)=\alpha$, or $\phi(\alpha)=\infty$ (Proposition~\ref{thm:positive}).

Another natural question is whether a continuous self-map of $\KS$ can move uncountably many limit ordinals while not mapping
them to the infinity point. We answer the latter affirmatively, showing that given a club subset $F$ of the limit ordinals
$L(\omega_1)$ then on the nonstationary set $L(\omega_1)\setminus F$ one is reasonably free to choose the value of a continuous
map (Theorem~\ref{thm:main}).

\end{section}

\begin{section}{Main results}\label{sec:main}

We denote by $L(\omega_1)$ the subset of $\omega_1$ consisting of limit ordinals and we set $S(\omega_1)=\omega_1\setminus L(\omega_1)$.
By a {\em ladder system\/} $\mathcal S$ on $\omega_1$ we mean a family $\mathcal S=(s_\alpha)_{\alpha\in L(\omega_1)}$
where, for each limit ordinal $\alpha$, we have $s_\alpha=\big\{s^n_\alpha:n\in\omega\big\}$ and $(s^n_\alpha)_{n\in\omega}$ is a strictly
increasing sequence in $S(\omega_1)$ order-converging to $\alpha$. The {\em ladder system topology\/} $\tauS$
on $\omega_1$ is the one for which the elements of $S(\omega_1)$ are isolated and the fundamental neighborhoods of a limit
ordinal $\alpha$ are unions of $\{\alpha\}$ with sets cofinite in $s_\alpha$. Then $S(\omega_1)$ is a discrete open subset
of $\omegaS=(\omega_1,\tauS)$ and $L(\omega_1)$ is a discrete closed subset of $\omegaS$. The topology $\tauS$ is finer
than the order topology. The relatively compact
subsets of $\omegaS$ are those which are almost contained in a finite union of ladders $s_\alpha$ (with ``almost''
meaning ``except for a finite subset of $\omega_1$''). The sequence $(s^n_\alpha)_{n\in\omega}$ converges to $\alpha$ in $\omegaS$
and a map $f$ defined in $\omegaS$ (taking values in an arbitrary space) is continuous if and only if $\big(f(s^n_\alpha)\big)_{n\in\omega}$ converges to $f(\alpha)$, for all $\alpha\in L(\omega_1)$. The space $\omegaS$ is locally
compact Hausdorff of height $2$ and we denote by $\KS=\omegaS\cup\{\infty\}$ its one-point compactification (which is
compact Hausdorff of height $3$). Since every point of $\omegaS$ has a countable neighborhood and compact subsets of $\omegaS$
are countable, a continuous map $\phi:\KS\to\KS$ with $\phi(\infty)\ne\infty$ has countable range. We are thus interested
in maps $\phi$ fixing the point $\infty$.

In what follows, for a subset of $\omega_1$, {\em club\/} is the abbreviation for {\em closed unbounded}.
\begin{prop}\label{thm:positive}
Let $\phi:\KS\to\KS$ be a continuous map with $\phi(\infty)=\infty$. Then there exists a club $F\subset\omega_1$
such that for all $\alpha\in F$ either $\phi(\alpha)=\alpha$, or $\phi(\alpha)=\infty$.
\end{prop}
\begin{proof}
For $\alpha\in\omega_1$, the level set $\phi^{-1}(\alpha)$ is a compact subset of $\omegaS$ and therefore countable.
Thus the Pressing Down Lemma (\cite[Lemma~II.6.15]{Kunen}) implies that the set $\big\{\alpha\in\omega_1:\phi(\alpha)<\alpha\big\}$ is nonstationary.
Let $F_1$ be a club disjoint from the latter. Consider the map $\psi:\omega_1\to\omega_1$ given by $\psi(\alpha)=\phi(\alpha)$
if $\phi(\alpha)\ne\infty$ and $\psi(\alpha)=0$ otherwise. By standard arguments, the set $F_2$ of ordinals $\alpha$
invariant by $\psi$ (i.e., $\beta<\alpha$ implies $\psi(\beta)<\alpha$) is club. For $\alpha\in F_2\cap L(\omega_1)$, we have that
$\phi[s_\alpha]=\big\{\phi(s^n_\alpha):n\in\omega\big\}$ is contained in the closed set $[0,\alpha]\cup\{\infty\}$ and therefore also $\phi(\alpha)$ is in
$[0,\alpha]\cup\{\infty\}$. The conclusion is obtained by taking $F=F_1\cap F_2\cap L(\omega_1)$.
\end{proof}

Proposition~\ref{thm:positive} leads to the question, whether, for $\phi$ having uncountable range, the set of fixed points of $\phi$ contains a club.
The following example shows that this is not the case.
\begin{example}\label{exa:movebastante}
Let $\phi:\KS\to\KS$ be a map whose restriction to $S(\omega_1)$ is an injection into $L(\omega_1)$ and that maps
every element of $L(\omega_1)\cup\{\infty\}$ to $\infty$. For limit $\alpha\in\omega_1$, the sequence $\big(\phi(s^n_\alpha)\big)_{n\in\omega}$
is injective and contained in $L(\omega_1)$; therefore it converges to $\infty=\phi(\alpha)$. This proves the continuity
of $\phi$ at the points of $\omega_1$. It is easy to see that a map $\psi:\KS\to\KS$ with $\psi(\infty)=\infty$ is continuous
at $\infty$ if and only if $\psi^{-1}(\alpha)$ is relatively compact in $\omegaS$ for all $\alpha\in\omega_1$ and also $\psi^{-1}[s_\alpha]$
is relatively compact in $\omegaS$ for all $\alpha\in L(\omega_1)$. It follows from this criterion that $\phi$ is continuous at $\infty$.
\end{example}

Proposition~\ref{thm:positive} says that continuous self-maps of $\KS$ (fixing $\infty$) are highly constrained in a club
subset $F$: namely every point of $F$ is either fixed or mapped to $\infty$. We now show that on the nonstationary
set $L(\omega_1)\setminus F$ a continuous self-map of $\KS$ can be chosen with a lot of freedom.
\begin{teo}\label{thm:main}
Let $\xi:A\to L(\omega_1)$ be an injective map defined in a nonstationary subset $A$ of $L(\omega_1)$. Then $\xi$ extends
to a continuous self-map $\phi$ of $\KS$ that maps $\infty$ and every element of $L(\omega_1)\setminus A$ to $\infty$.
\end{teo}
\begin{proof}
We claim that the ladders $s_\alpha$, $\alpha\in A$, can be made disjoint by removing a finite number of elements from each of them.
Obviously, this will not change $\tauS$. Since the ladders are already almost disjoint, the claim is trivial if $A$ is countable.
In the rest of the proof of the claim, we assume that $A$ is uncountable.
Let $F\subset\omega_1$ be a club disjoint from $A$ and, for $\alpha\in A$, let $\lambda(\alpha)$ be the largest element of
$F\cup\{0\}$ below $\alpha$. Then $\lambda:A\to\omega_1$ is regressive (i.e., $\lambda(\alpha)<\alpha$, for all $\alpha\in A$) and
cofinal (i.e., for all $\alpha\in\omega_1$, there exists $\beta_0\in A$, such that, for all $\beta\in A$,
$\beta\ge\beta_0$ implies $\lambda(\beta)\ge\alpha$). For $\alpha\in A$, remove the finite number of elements of $s_\alpha$ strictly below
$\lambda(\alpha)$; now $s_\alpha$ denotes the set thus obtained. Consider the binary relation $R$ on $A$ defined by $(\alpha,\beta)\in R\Leftrightarrow s_\alpha\cap s_\beta\ne\emptyset$.
The fact that $\lambda$ is cofinal implies that, for each $\alpha\in A$, only a countable number of elements of $A$ are $R$-related to $\alpha$.
Hence the equivalence classes defined by the equivalence relation spanned by $R$ (i.e., the smallest equivalence relation containing $R$) are also countable.
The proof of the claim is now obtained by using the
fact that a countable number of almost disjoint sets can be made disjoint by removing a finite number of elements from each of them.

The ladders $s_\alpha$, $\alpha\in A$, are now assumed to be disjoint. Define $\phi$ by mapping $\alpha$ and
each element of $s_\alpha$ to $\xi(\alpha)$, for $\alpha\in A$; map everything else to $\infty$. The continuity
of $\phi$ at the points of $A\cup S(\omega_1)$ is clear. For $\alpha\in L(\omega_1)\setminus A$, the sequence $\big(\phi(s^n_\alpha)\big)_{n\in\omega}$ is contained in $L(\omega_1)\cup\{\infty\}$ and has no constant subsequence
contained in $L(\omega_1)$. Therefore, it converges to $\infty=\phi(\alpha)$. Finally, the continuity of $\phi$ at $\infty$
follows from the criterion explained in Example~\ref{exa:movebastante}.
\end{proof}

In the statement of Theorem~\ref{thm:main} the assumptions on $\xi$ can be weakened: namely, it suffices to assume that
$\xi:A\to L(\omega_1)\cup\{\infty\}$ be a map with $\xi^{-1}(\alpha)$ finite, for all $\alpha\in L(\omega_1)$. Note that this weaker assumption on $\xi$ is necessary: if $\alpha\in L(\omega_1)$ then $\phi^{-1}(\alpha)$
is a compact subset of $\omegaS$ and hence $\phi^{-1}(\alpha)\cap L(\omega_1)$ must be finite.

\end{section}

\begin{section}{Operators on $C(\KS)$}\label{sec:operators}

We now discuss the bounded operators on the Banach space $C(\KS)$ of continuous real-valued maps on $\KS$. Given a compact Hausdorff space $K$ and $g\in C(K)$, we denote by $M_g$ the multiplication operator $f\mapsto fg$
on $C(K)$. If $g$ belongs to the space $\mathfrak B(K)$ of real-valued bounded Borel functions on $K$, we denote by $M_g^*$ the operator on $C(K)^*$
defined by $\mu\mapsto\int g\,\dd\mu$. (When $g$ is continuous, $M_g^*$ is indeed the adjoint of $M_g$.) In \cite{Survey}, an operator $T$ on $C(K)$
is called a {\em weak multiplication\/} if $T+M_g$ is weakly compact for some $g\in C(K)$ and it is called a {\em weak multiplier\/} if
$T^*+M_g^*$ is weakly compact, for some $g\in\mathfrak B(K)$. By Gantmacher's Theorem (\cite[Theorem~VI.4.8]{DS}), the adjoint of a weakly compact operator is weakly compact
and therefore every weak multiplication is a weak multiplier.
In \cite{Survey}, the space $C(K)$ is said to have {\em few operators\/} if every operator is a weak multiplier.
If $K$ is infinite and scattered then $C(K)$ cannot have few operators
in this sense: namely, in this case $C(K)$ has a complemented copy of $c_0$ (\cite[Theorem~14.26]{Fabian}) and thus is isomorphic to its closed hyperplanes.
But this is impossible if $C(K)$ has few operators (\cite[Theorem~3.2]{Survey}). However, in \cite{PZ}, it is presented (under $\clubsuit$) an example
of a nonmetrizable scattered compact space $K$ (of infinite height) for which every operator on $C(K)$ is a constant multiple of the identity plus an operator of separable range. Having this in mind, for a given compact Hausdorff space $K$, one can
consider the question of whether every operator $T$ on $C(K)$ satisfies one of the following conditions:
\begin{itemize}
\item[(a)] $T+M_g$ has separable range, for some $g\in C(K)$;
\item[(b)] $T^*+M_g^*$ has separable range, for some $g\in\mathfrak B(K)$;
\item[(c)] the restriction of $(T^*+M_g^*)^*$ to $C(K)$ has separable range, for some $g\in\mathfrak B(K)$.
\end{itemize}
Condition (a) is a natural modification of the notion of weak multiplication and (b), (c) are natural candidates for modifications
of the notion of weak multiplier. Note that (a) implies (c). However, (a) does not imply (b), as shown in Example~\ref{exa:referee} below.
The following lemma gives a simple sufficient condition for a composition operator $\C_\phi$ not to satisfy (b).
\begin{lem}\label{thm:lemanovo}
Let $K$ be a compact Hausdorff space and $\phi:K\to K$ be a continuous map. If there exists an uncountable subset $A$ of $K$ such that
$\phi\vert_A$ is injective and $\phi(A)\cap A=\emptyset$ then $T=\C_\phi$ does not satisfy condition (b).
\end{lem}
\begin{proof}
Simply notice that if $g\in\mathfrak B(K)$ then $T^*+M_g^*$ maps $\big\{\delta_x:x\in A\big\}$ to an uncountable discrete set,
where $\delta_x$ denotes the delta-measure supported at $x$.
\end{proof}

\begin{example}\label{exa:referee}
If $K=S^1$ is the unit circle and $\phi:K\ni x\mapsto -x\in K$ is the antipodal map then $T=\C_\phi$ clearly satisfies (a), because $C(K)$ is separable.
However, using Lemma~\ref{thm:lemanovo} with $A$ equal to an open half circle, we obtain that $T$ does not satisfy (b).
\end{example}

We now prove the main result of this section.
\begin{teo}
If $K=\KS$ then there exists an operator $T$ on $C(K)$ that does not satisfy either (b) or (c) (and hence does not
satisfy (a)). More specifically, if $\phi$ is defined as in Example~\ref{exa:movebastante}, then $T=\C_\phi$ does not satisfy either (b) or (c).
\end{teo}
\begin{proof}
The fact that $T$ does not satisfy (b) follows by using Lemma~\ref{thm:lemanovo} with $A=S(\omega_1)$.
Let us prove that $T$ does not satisfy (c). Note first that, for $g\in\mathfrak B(\KS)$, the restriction of $(T^*+M_g^*)^*$
to $C(\KS)$ is identified with $T+M_g:C(\KS)\to\mathfrak B(\KS)$, where $\mathfrak B(\KS)$ is identified with a subspace
of $C(\KS)^{**}$ in the natural way. Assuming that $\C_\phi+M_g$ has separable range, we prove first that $g$ must vanish outside
a countable set. If there were uncountably many $\alpha\in S(\omega_1)$ with $g(\alpha)\ne0$, there would exist $\varepsilon>0$
and an uncountable subset $A$ of $S(\omega_1)$ with $\vert g(\alpha)\vert\ge\varepsilon$, for all $\alpha\in A$. Then
$\C_\phi+M_g$ would map $\big\{\chi_{\{\alpha\}}:\alpha\in A\big\}$ to an uncountable discrete set (where $\chi$
denotes characteristic function). Similarly, if there were uncountably many $\alpha$ in $L(\omega_1)$ with
$g(\alpha)\ne0$, there would exist $\varepsilon>0$ and an uncountable subset $A$ of $L(\omega_1)$ with $\vert g(\alpha)\vert\ge\varepsilon$,
for all $\alpha\in A$. Then $\C_\phi+M_g$ would map the set $\big\{\chi_{s_\alpha\cup\{\alpha\}}:\alpha\in A\big\}$ to an uncountable discrete set. We have thus proven that the support of $g$, $\supp g$, must be countable. This implies that $M_g$ has separable range,
since it factorizes through the restriction map $C(\KS)\to C(\supp g)$ and $C(\supp g)$ is separable. Finally,
the fact that $M_g$ has separable range implies that $\C_\phi$ has separable range as well.
This is a contradiction, because $\C_\phi\big[C(\KS)\big]\equiv C\big(\phi[\KS]\big)$ and $\phi[\KS]$ contains an uncountable discrete set.
\end{proof}

\noindent\textbf{Acknowledgments.}\enspace The authors wish to thank the anonymous referee for the valuable suggestions.

\end{section}

\end{document}